\documentclass[11pt,a4paper]{article}
\usepackage{amsmath,amssymb,amsthm,bbm,color,url}
\usepackage[margin=2cm]{geometry}
\usepackage{ytableau}
\usepackage{graphicx}
\usepackage[margin=1.2cm,font=small,labelfont=bf]{caption}
\usepackage {tikz}
\usetikzlibrary{matrix,calc,backgrounds,arrows,shapes,patterns,arrows.meta}

\newcommand{\wedgepower}{{\ensuremath{\bigwedge}}}
\newcommand{\IC}{\mathbb{C}}

\newcommand{\IN}{\mathbb{N}}

\newcommand{\HWV}{\textup{HWV}}
\newcommand{\GL}{\textup{GL}}
\newcommand{\diag}{\textup{diag}}

\newcommand{\Specht}[1]{\mathbb{S}_{#1}}
\newtheorem{lemma}{Lemma}
\newtheorem{cor}{Corollary}

\newtheorem{theorem}{Theorem}

\newcommand{\la}{\lambda}
\newcommand{\al}{\alpha}
\newcommand{\be}{\beta}
\newcommand{\ga}{\gamma}
\newcommand{\plp}{\mathbin{\diamond}}

\newcommand{\g}{\textup{\texttt{k}}}
\newcommand{\rg}{\overline{\g}}
\newcommand{\sgn}{\mathop{\mathrm{sgn}}}
\newcommand{\id}{\textup{id}}

\def\P{{{\rm{\textsf{P}} }}}

\def\SP{\textup{{\rm{\textsf{\#P}}}}}

\def\NP{\textup{{{\rm{\textsf{NP}}}}}}

\newcommand{\tikzbox}[3]{
 \fill[very nearly transparent] (#1,#2,#3+1) -- (#1+1,#2,#3+1) -- (#1+1,#2,#3) -- (#1+1,#2+1,#3) -- (#1,#2+1,#3) -- (#1,#2+1,#3+1) -- cycle;

 \draw (#1,#2,#3) -> (#1+1,#2,#3);
 \draw (#1+1,#2,#3) -> (#1+1,#2+1,#3);
 \draw (#1+1,#2+1,#3) -> (#1,#2+1,#3);
 \draw (#1,#2+1,#3) -> (#1,#2,#3);

 \draw (#1,#2,#3+1) -> (#1+1,#2,#3+1);
 \draw (#1+1,#2,#3+1) -> (#1+1,#2+1,#3+1);
 \draw (#1+1,#2+1,#3+1) -> (#1,#2+1,#3+1);
 \draw (#1,#2+1,#3+1) -> (#1,#2,#3+1);

 \draw (#1,#2,#3) -> (#1,#2,#3+1);
 \draw (#1,#2+1,#3+1) -> (#1,#2+1,#3);
 \draw (#1+1,#2,#3) -> (#1+1,#2,#3+1);
 \draw (#1+1,#2+1,#3+1) -> (#1+1,#2+1,#3);
}

\newcommand{\tikzcuboid}[7]{
 \fill[nearly transparent, #7] (#1,#2,#6+1) -- (#4+1,#2,#6+1) -- (#4+1,#2,#3) -- (#4+1,#5+1,#3) -- (#1,#5+1,#3) -- (#1,#5+1,#6+1) -- cycle;

 \draw (#1,#2,#3) -> (#4+1,#2,#3);
 \draw (#4+1,#2,#3) -> (#4+1,#5+1,#3);
 \draw (#4+1,#5+1,#3) -> (#1,#5+1,#3);
 \draw (#1,#5+1,#3) -> (#1,#2,#3);

 \draw (#1,#2,#6+1) -> (#4+1,#2,#6+1);
 \draw (#4+1,#2,#6+1) -> (#4+1,#5+1,#6+1);
 \draw (#4+1,#5+1,#6+1) -> (#1,#5+1,#6+1);
 \draw (#1,#5+1,#6+1) -> (#1,#2,#6+1);

 \draw (#1,#2,#3) -> (#1,#2,#6+1);
 \draw (#1,#5+1,#6+1) -> (#1,#5+1,#3);
 \draw (#4+1,#2,#3) -> (#4+1,#2,#6+1);
 \draw (#4+1,#5+1,#6+1) -> (#4+1,#5+1,#3);
}

\tikzset{coordinateIIID/.style={x={(240:0.8cm)}, y={(-10:1cm)}, z={(0,1cm)}, scale=0.54}}

\newcommand{\grid}{
 \node at (1,7.75,1) {\footnotesize\hspace{1.4cm} 1$^{\textup{st}}$ coord.};
 \node at (1,1,8.1) {\footnotesize 3$^{\textup{rd}}$ coord.};
 \node at (8,1,1) {\footnotesize 2$^{\textup{nd}}$ coord.};
 \draw[-Triangle] (1,1,1) -- (7.75,1,1);
 \draw[-Triangle] (1,1,1) -- (1,7.75,1);
 \draw (1,7,1) -- (7,7,1);
 \draw (7,1,1) -- (7,7,1);
 
 \draw[-Triangle] (1,1,1) -- (1,1,7.75);
 \draw (1,1,7) -- (1,7,7);
 \draw (1,7,7) -- (1,7,1);
 \draw (1,7,1) -- (1,1,1); 

 \draw (1,1,1) -- (7,1,1);
 \draw (7,1,1) -- (7,1,7);
 \draw (7,1,7) -- (1,1,7);
 \draw (1,1,7) -- (1,1,1);

 \foreach \i in {2,...,6}\draw (\i,0.9,1) -- (\i,1.1,1);
 \foreach \i in {2,...,6}\draw (\i,0.9,7) -- (\i,1.1,7);
 \foreach \i in {2,...,6}\draw (0.9,\i,1) -- (1.1,\i,1);
 \foreach \i in {2,...,6}\draw (0.9,\i,7) -- (1.1,\i,7);
 \foreach \i in {2,...,6}\draw (0.9,7,\i) -- (1.1,7,\i);
 \foreach \i in {2,...,6}\draw (0.9,1,\i) -- (1.1,1,\i);
 \foreach \i in {2,...,6}\draw (1,0.9,\i) -- (1,1.1,\i);
 \foreach \i in {2,...,6}\draw (7,0.9,\i) -- (7,1.1,\i);
 \foreach \i in {2,...,6}\draw (\i,6.9,1) -- (\i,7.1,1);
 \foreach \i in {2,...,6}\draw (6.9,\i,1) -- (7.1,\i,1);
}

\title{All Kronecker coefficients are reduced Kronecker coefficients}
\author{Christian Ikenmeyer\thanks{University of Warwick, christian.ikenmeyer$@$warwick.ac.uk, supported by EPSRC grant EP/W014882/1}\ \ and Greta Panova\thanks{University of Southern California, gpanova$@$usc.edu, partially supported by NSF CCF:AF grant 2007652}}

\begin{document}
\raggedbottom

\maketitle

\begin{abstract}
We settle the question of where exactly do the reduced Kronecker coefficients lie on the spectrum between the Littlewood-Richardson and Kronecker coefficients by
showing that every Kronecker coefficient of the symmetric group is equal to a reduced Kronecker coefficient by an explicit construction.
This implies the equivalence of a question by Stanley from 2000 and a question by Kirillov from 2004 about combinatorial interpretations of these two families of coefficients.
Moreover, as a corollary, we deduce that deciding the positivity of reduced Kronecker coefficients is $\NP$-hard, and computing them is $\SP$-hard under parsimonious many-one reductions. 
\end{abstract}

\section{Introduction}

The \emph{Kronecker coefficients} $\g(\la,\mu,\nu)$ of the symmetric group $S_n$ are some of the most classical, yet largely mysterious, quantities in Algebraic Combinatorics and Representation Theory. The Kronecker coefficient is the multiplicity of the irreducible $S_n$ representation $\mathbb{S}_{\nu}$ in the tensor product $\mathbb{S}_\la \otimes \mathbb{S}_\mu$ of two other irreducible $S_n$ representations. Murnaghan defined them in 1938 as an analogue of the Littlewood-Richardson coefficients $c^{\la}_{\mu\nu}$ of the general linear group $\GL_N$, which are the multiplicity of the irreducible Weyl modules $V_\la$ in the tensor products $V_{\mu} \otimes V_{\nu}$. Yet, the analogy has not translated far into their properties. The Littlewood-Richardson coefficients have a beautiful positive combinatorial interpretation and their positivity is ``easy'' to decide, formally it is in $\P$. However, positive combinatorial formulas for the Kronecker coefficients have eluded us so far, see Section~\ref{ss:related},
 and their positivity is hard to decide.

The \emph{reduced Kronecker coefficients} $\rg(\al,\be,\ga)$ are defined as the stable limit of the ordinary Kronecker coefficients
\[\rg(\al,\be,\ga) := \lim_{n \to \infty} \g(\, (n-|\al|,\al), \ (n-|\be|,\be), \ (n-|\ga|,\ga) \,).\]
These coefficients are called \emph{extended Littlewood-Richardson numbers} in \cite{Kir}, since in the special case when $|\al|=|\be|+|\ga|$
we have $\rg(\al,\be,\ga)=c^{\al}_{\be,\ga}$, the Littlewood-Richardson coefficient.
Problem~2.32 in \cite{Kir} asks for a combinatorial interpretation of $\rg(\al,\be,\ga)$.
As such they have been considered as an intermediate, an interpolation, between the Littlewood-Richardson and Kronecker coefficients. They have been an object of independent interest, see~\cite{Mur38,Mur56,Bri93,Val99,Kir,BOR2,BDO,CR,Manivel,SS,IP17,PPr,OZ21}, and considered better behaved than the ordinary Kronecker coefficients. 

This is, however, not the case. As we show, every Kronecker coefficient is equal to an explicit reduced Kronecker coefficient of not much larger partitions. 
In particular we prove the following theorem.
\begin{theorem}
\label{thm:main}
For all partitions $\la$, $\mu$, $\nu$ of equal sizes, we have
\[
\g(\la,\mu,\nu) \ = \ 
\rg\big(\,\nu_1^{\ell(\la)}+\la , \ \nu_1^{\ell(\mu)}+\mu , \ (\nu_1^{\ell(\la)+\ell(\mu)},\nu)\,\big).
\]
\end{theorem}
Here $a^b := (\underbrace{a,\ldots,a}_{b \text{ many}})$ and
$(\nu_1^b,\nu) := (\underbrace{\nu_1,\ldots,\nu_1}_{b \text{ many}},\nu_1,\nu_2,\nu_3,\ldots)$.

Theorem~\ref{thm:main} implies that in a very strong sense, on the spectrum between Littlewood-Richardson and Kronecker coefficients, the reduced Kronecker coefficients are at the same point as the ordinary Kronecker coefficients.
In particular, Theorem~\ref{thm:main} implies that Problem~2.32 in \cite{Kir} is equivalent to Problem~10 in \cite{Sta00}:
Finding a combinatorial interpretation for the Kronecker coefficient or for the reduced Kronecker coefficient are the same problem. Formally, Conjecture 9.1 and 9.4 in \cite{Pak22} are the same.

Our result can be interpreted in a positive or in a negative way.
On the one hand, the reduced Kronecker coefficients cannot be easier to understand than the ordinary Kronecker coefficients. On the other hand, understanding the reduced Kronecker coefficients is sufficient to understand all ordinary Kronecker coefficients.
As a corollary, we settle the conjecture from~\cite[\S 4.4]{PPr} on the hardness of deciding the positivity of $\rg(\alpha,\beta,\gamma)$.

\begin{cor}[settles conj.~in \protect{\cite[\S 4.4]{PPr}}]
Given $\al,\be,\ga$ in unary, deciding if $\rg(\al,\be,\ga)>0$ is $\NP$-hard.
\end{cor}
\begin{proof}
This follows directly from Theorem~\ref{thm:main} and the fact that deciding $\g(\la,\mu,\nu)>0$ is $\NP$-hard~\cite{IMW17}.
\end{proof}
Moreover, by the same immediate argument it is now clear that computing the reduced Kronecker coefficient is strongly \#P-hard under parsimonious many-one reductions (whereas, the argument in \cite{PPr} gives only the \#P-hardness under Turing reductions).

\subsection{Background and definitions}
We refer to~\cite{JK,EC2,Sag} for basic definitions and properties from Algebraic Combinatorics and Representation Theory, and include the main definitions here for completeness.

We write $[a,b]:=\{a,a+1,\ldots,b\}$, and $[n]:=[1,n]$.
A \emph{composition} of $n$ is a sequence of nonnegative integers that sum up to $n$.
A \emph{partition} $\la=(\la_1,\la_2,\ldots)$ of $n$, denoted $\la \vdash n$, is a weakly decreasing composition.
Its size is $|\la| := \sum_i \la_i$.
Denote by $\ell(\la)=\max\{i \mid \la_i > 0\}$ the \emph{length} of~$\la$.
We interpret $\la$ as a vector of arbitrary length $\geq \ell(\la)$ by appending zeros.
We denote by $(n)$ the partition of $n$ of length~1.
To every partition we associate its \emph{Young diagram}, which is a list of left-justified rows of boxes, $\la_i$ many boxes in row $i$.
We write $\la'$ do denote the \emph{transpose} partition, i.e., the partition that arises from reflecting the Young diagram at the main diagonal. Formally, $\la'_j := \max\{i \mid \la_i\geq j\}$.
We add partitions row-wise: $(\la+\mu)_i=\la_i+\mu_i$.
We define $\la\plp\mu := (\la'+\mu')'$.
Note that $\plp$ is commutative and associative, and that
if $\la_{\ell(\la)}\geq \mu_1$, then $\la\plp\mu=(\la_1,\ldots,\mu_1,\ldots)$ is just the concatenation of rows.
The Specht modules $\Specht\la$ for $\la \vdash n$ are the irreducible representation of the symmetric group $S_{n}$, see~\cite{JK,EC2,Sag}.

The \emph{Kronecker coefficient} $\g(\la,\mu,\nu)$ is the structure constant\footnote{We remark that in the combinatorics literature these coefficients have usually been denoted by $g$, e.g. $g(\la,\mu,\nu)$, but here we use $\g$ to avoid overlap with the notation used for the Representation Theory of $GL_N$. } defined via
\[
\chi^\mu \cdot \chi^\nu = \sum_\la \g(\la,\mu,\nu) \chi^\la,
\]
or equivalently via Specht modules as
\[
\Specht\nu \otimes \Specht\mu = \sum_\la \Specht\la^{\oplus \g(\la,\mu,\nu)}.
\]
From this description, it is immediately clear that $\g(\la,\mu,\nu)$ is a nonnegative integer.
The problem of finding a combinatorial interpretation of $\g(\la,\mu,\nu)$ is wide open \cite{Sta00,IP22,Pan23}. 

The Kronecker coefficients were defined by Murnaghan~\cite{Mur38} in 1938 as the analogues of the \emph{Littlewood-Richardson coefficients} $c^{\la}_{\mu\nu}$. These are the structure constants in the ring of irreducible $\GL_N$ representations, the Weyl modules $V_{\la}$, and are formally given by
$$V_\mu \otimes V_\nu = \bigoplus_\la V_{\la}^{\oplus c^{\la}_{\mu\nu}}.$$

Some simple properties, see~\cite{JK,Sag} include the transposition invariance $\g(\la,\mu,\nu) = \g(\la',\mu',\nu)$, since $\Specht{1^n}\otimes\Specht\la=\Specht{\la'}$~\cite{JK}. From their definition, and the fact that $\chi^\la(\pi) \in \mathbb{Z}$, see~\cite{JK,Sag}, we have
\[
\g(\la,\mu,\nu)=\tfrac{1}{n!}\sum_{\pi\in S_n}\chi^\la(\pi)\chi^\mu(\pi)\chi^\nu(\pi),
\]
and thus we have the $S_3$ invariance $\g(\la,\mu,\nu)=\g(\la,\nu,\mu)=\g(\mu,\nu,\la)=\cdots$.
Note that the Kronecker coefficient is not invariant under transposing an odd number of partitions, and
we define 
$$\g'(\la,\mu,\nu) := \g(\la',\mu',\nu') = \g(\la',\mu,\nu) = \g(\la,\mu',\nu) = \g(\la,\mu,\nu').$$
It is known that $\g(\la,\mu,\nu)=0$ if $\ell(\la)>\ell(\mu)\cdot\ell(\nu)$ \cite{Dvi93}, which also follows by combining
$\g(\la,\mu,\nu)=\g(\la,\mu',\nu')$ with Lemma~\ref{lem:upperbound}.
We define the \emph{stable range} as the set of triples $(\la,\mu,\nu)$ that satisfy
\[
\forall i\geq 0: \ \g(\la,\mu,\nu) \ = \ \g\big(\, \la+(i), \ \mu+(i) , \ \nu+(i)\,\big).
\]
There are several proofs for the fact that for arbitrary $(\al,\be,\ga)$ with $|\al|=|\be|=|\ga|$, the triple
$(\al+(i), \ \be+(i) ,\ga+(i))$ is in the stable range for $i$ large enough (and hence for all $i$ from then on), and upper bounds on the necessary $i$ are known, see e.g.\ \cite{Bri93}, \cite{Val99}, \cite{BOR2}, \cite[\S7.4]{Ike12}, \cite{PP14}.
The reduced Kronecker coefficient is defined as this limit value: 
$$\rg(\al,\be,\ga) := \lim_{n \to \infty} \g(\, (n-|\al|,\al), \ (n-|\be|,\be), \ (n-|\ga|,\ga) \,)$$
 for arbitrary partitions $\al$, $\be$, $\ga$ (in particular, we do \emph{not} require $|\alpha|=|\beta|=|\gamma|$).

\subsection{Related work}
\label{ss:related}
The Littlewood-Richardson coefficients can be computed by the Littlewood-Richardson rule, stated in 1934 and proven formally about 40 years later. It says that $c^{\la}_{\mu\nu}$ is equal to the number of LR tableaux of shape $\la/\mu$ and content $\nu$, see Section~\ref{ss:sym_fun_background} and~\cite{EC2,Sag}.
The apparent analogy in definitions motivates the community to search for such interpretations for the Kronecker coefficients.
Interest in efficient ways to compute $\g(\la,\mu,\nu)$ and $\rg(\al,\be\ga)$ dates back at least to Murnaghan \cite{Mur38}.
Specific interest in nonnegative combinatorial interpretations of $\g(\la,\mu,\nu)$ can be found in \cite{Las79,GR85}, and was formulated clearly again by Stanley as Problem 10 in  his list ``Open Problems in Algebraic Combinatorics''~\cite{Sta00}:

\medskip

``Find a combinatorial interpretation of the Kronecker product coefficients $\g(\la,\mu,\nu)$, thereby combinatorially reproving that they are nonnegative.''

\medskip
See also~\cite{Pan23} for a detailed discussion on this topic.

Despite its natural and fundamental nature and the variety of efforts, this question has seen relatively little progress.
In 1989 Remmel found a combinatorial rule for $\g(\la,\mu,\nu)$ when two of the partitions are hooks \cite{remm:89}. In 1994 Remmel and Whitehead~\cite{RW94} found $\g(\la,\mu,\nu)$ for  $\ell(\la), \ell(\mu)\leq 2$, which was subsequently studied also in~\cite{BMS}. In 2006 Ballantine and Orellana~\cite{BO06} established a rule for $\g(\la,\mu,\nu)$ when  $\mu=(n-k,k)$ and $\la_1 \geq 2k-1$. The most general rule for $\nu=(n-k,1^k)$, a hook, and any other two partitions, was estabslished by Blasiak in 2012~\cite{Bla}, and later simplified in~\cite{Liu17,BL}. Other special cases include multiplicity-free Kronecker products by Bessenrodt-Bowman~\cite{BB17}, triples of partitions which are marginals of pyramids by Ikenmeyer-Mulmuley-Walter~\cite{IMW17}, $\g(m^k,m^k,(mk-n,n))$ as counting labeled trees by Pak-Panova~\cite[slide~9]{Pporto}, near-rectangular partitions by Tewari in~\cite{T15}, etc. As shown in~\cite{IMW17}, computing the Kronecker coefficients is $\SP$-hard, and deciding positivity is $\NP$-hard.

It was shown by Murnaghan~\cite{Mur56} that the reduced Kronecker coefficients generalize the Littlewood-Richardson coefficients as
$$
\rg(\al,\be,\ga) \, = \, c^\al_{\be\ga} \quad \text{for} \quad |\al| =  |\be|  + |\ga|,
$$
which motivates Kirillov's naming of $\rg$ as ``extended Littlewood-Richardson coefficients''. This relationship and other properties have motivated an independent interest in the reduced Kronecker coefficients as intermediates between Littlewood-Richardson and ordinary Kronecker coefficients.  Some special cases of combinatorial interpretations can be derived from the existing ones for the ordinary Kronecker coefficients. In~\cite{CR} a combinatorial interpretation was given when $\mu,\nu$ are rectangles and $\la$ is one row.  A combinatorial interpretation of $\rg(\al,\be,\ga)$ in the subcase where $\ell(\al)=1$ was obtained in~\cite{BO05, BO06}.  Methods to compute them have been discussed in \cite{Mur38,Mur56} and have been developed in a series of papers, see~\cite{BOR2,BDO,OZ20,OZ21}.
As observed in~\cite{BDO} the reduced Kronecker coefficients are also the structure constants for the ring of so called character polynomials~\cite{Mac}. The reduced Kronecker coefficients are a special case of a more general stability phenomenon that if $\g(k\al,k\be,k\ga)=1$  for all $k$ then $\g(\la+N\al,\mu+N\be,\nu+N\ga)$ stabilizes as $N\to \infty$, see~\cite{Ste,SS,Val20}.

The  Kronecker coefficients can be expressed as a small alternating sum of reduced Kronecker coefficients, and reduced Kronecker coefficients are certain sums of ordinary Kronecker coefficients for smaller partitions, see~\cite{BOR2}.
These relationships showed that reduced Kronecker coefficients are also \SP-hard to compute, see~\cite{PPr}. However, these relations did not imply that deciding  positivity of reduced Kronecker coefficients is \NP-hard.

It is important to note that deciding if $c^{\la}_{\mu\nu}>0$ is in $\P$, since they count integer points in a polytope that has an integral vertex whenever it is nonempty. This was proved by Knutson and Tao in~\cite{KT} as part of their proof of the Saturation theorem for Littlewood-Richardson coefficients, namely that
$c^{N\la}_{N\mu, N\nu}>0 \Longleftrightarrow c^{\la}_{\mu\nu}>0$.
It is known that the Kronecker coefficients (and hence also the reduced Kronecker coefficients) satisfy the so-called \emph{semigroup property} \cite[Thm~2.7]{Chr06}, which implies that if $\g(\la,\mu,\nu)>0$, then $\forall N>0: \g(N\la,N\mu,N\nu)>0$,
and
if $\rg(\al,\be,\ga)>0$, then $\forall N>0: \g(N\al,N\be,N\ga)>0$.
The Kronecker coefficients do not satisfy the saturation property, because $\g(2^2,2^2,2^2)=1$, but $\g(1^2,1^2,1^2)=0$.
Until recently it was believed that the reduced Kronecker coefficients have the saturation property: It was conjectured in~\cite{Kir,Kly} that 
if $\rg(N\al,N\be, N\ga) > 0$ for some $N>0$, then
 $\rg(\al,\be,\ga) > 0 .$ This was disproved in~\cite{PPr} in 2020 and moved the reduced Kroneckers away from the Littlewood-Richardson on that spectrum.

\section{Setting up the proof of Theorem~\protect{\ref{thm:main}}}
\label{sec:main}
We discovered Theorem~\ref{thm:main} using the natural interpretation of $\g(\la,\mu,\nu)$ via the general linear group, see \S\ref{sec:gln}, and the relationship with 3-dimensional binary contingency arrays. 
We set the proof up in this section, reducing to a more general Theorem~\ref{thm:stable_range}, which has a short proof via $GL_N$ in  \S\ref{sec:gln}.
We also give two short, self-contained proofs using basic symmetric function techniques in \S\ref{sec:symfct}.
Once the statement of Theorem~\ref{thm:main} is known, it can also be readily deduced from a formula from 2011 by Briand, Orellana, and Rosas \cite{BOR2},
see the discussion at the end of Section~\ref{ss:sym_fun_proofs}.

\medskip

We prove a slightly stronger statement than Theorem~\ref{thm:main}: For $l\geq\ell(\la)$, $m\geq\ell(\mu)$, $c\geq\nu_1$ we have
\[
\g(\la,\mu,\nu) = 
\rg\big(\,c^{l}+\la , \ c^{m}+\mu , \ c^{l+m}\plp \nu\,\big).
\]
We start with a classical identity that can be proved in several ways, see e.g.\
\cite[Thm~2.4']{Dvi93},
\cite[Pf of Lem~2.1]{BOR09},
\cite[Thm~3.1]{Val09}, \cite[Cor.~4.4.14]{Ike12}.

\begin{lemma}
\label{lem:shift}
Let $\la,\mu,\nu$ be partitions with $\ell(\la)\leq l$, $\ell(\mu)\leq m$. Then
\[
\g(\la,\mu,\nu) \ = \ \g( \, m^l+\la, \ l^m+\mu, \ 1^{lm}+\nu \, ).
\]
\end{lemma}
The situation is depicted in Figure~\ref{fig:lem:shift}.
\begin{figure}[ht]
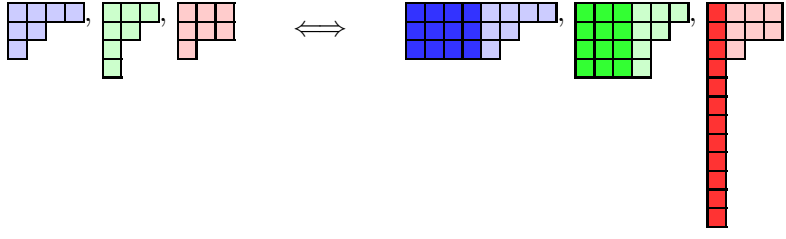

\ytableausetup{boxsize=1.4ex}

$$\ydiagram[*(blue!20)]{4,2,1}, \, \ydiagram[*(green!20)]{3,2,1,1},\, \ydiagram[*(red!20)]{3,3,1} \qquad \raisebox{-0.2cm}{$\Longleftrightarrow$} \qquad
\ydiagram[*(blue!80)]{4,4,4}*[*(blue!20)]{4+4,4+2,4+1} , \, \ydiagram[*(green!80)]{3,3,3,3}*[*(green!20)]{3+3,3+2,3+1,3+1}, \,
\ydiagram[*(red!80)]{1,1,1,1,1,1,1,1,1,1,1,1}*[*(red!20)]{1+3,1+3,1+1}$$
\caption{An example of the situation in Lemma~\ref{lem:shift} with $\la=(4,2,1)$, $\mu=(3,2,1,1)$, $\nu=(3,3,1)$, $l=3$, and $m=4$.}
\label{fig:lem:shift}
\end{figure}
Note that if $\ell(\nu)>lm$, then 
$\ell(\nu)>lm\geq\ell(\la)\cdot\ell(\mu)$, and hence $\g(\la,\mu,\nu)=0$.
Moreover, $\ell(1^{lm}+\nu)=\ell(\nu)>lm\geq\ell(\la)\cdot\ell(\mu) = \ell(m^l+\la)\cdot\ell(l^m+\mu)$, and hence $\g( \, m^l+\la, \ l^m+\mu, \ 1^{lm}+\nu \, )=0$.
So we can assume that $\ell(\nu)\leq lm$.
We give two proofs in this case, one in \S\ref{sec:gln} and one in~\S\ref{sec:symfct}.

The following Lemma~\ref{lem:walls} is proved by applying
Lemma~\ref{lem:shift} twice, in different directions. An illustration of the situation is given in Figure~\ref{fig:lem:walls}.

\begin{lemma}
\label{lem:walls}
Let $\la$, $\mu$, $\nu$ be partitions of the same size, and
let $l \geq \ell(\la)$, $m \geq \ell(\mu)$ and $c \geq \nu_1$.
Let $d=(m+1)c$, $e=(l+1)c$.
Then
\[
\g(\la,\mu,\nu) \ = \ \g\big(\, (d)\plp(c^l + \la), \ (e)\plp(c^m + \mu), \ c^{l+m+1} \plp \nu \,\big).
\]
\end{lemma}

\begin{figure}[ht]
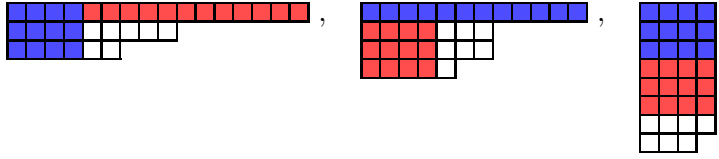

$$\ydiagram[*(blue!70)]{4,4,4}*[*(red!70)]{4+12}*{4+0,4+5,4+2} \,\, , \quad  \ydiagram[*(blue!70)]{12}*[*(red!70)]{0,4,4,4}*{0,4+3,4+3,4+1} \,\, , \quad
\ydiagram[*(blue!70)]{4,4,4}*[*(red!70)]{0,0,0,4,4,4}*{0,0,0,0,0,0,4,3}$$
\caption{An example of the proof of Lemma~\ref{lem:walls} with $\la=(5,2)$, $\mu=(3,3,1)$ and $\nu=(4,3)$, with $l=2$, $m=3$ and $c=4$. The red boxes are the addition from the first application of Lemma~\ref{lem:shift} and the blue boxes are the second application.}
\label{fig:lem:walls}
\end{figure}

\begin{proof}
We apply Lemma~\ref{lem:shift} twice as follows.

\mbox{~}

\vspace{-1.1cm}
\begin{eqnarray*}
\g(\la,\mu,\nu)
&=&
\g( \, \la', \, \mu, \, \nu' \, )
\\
&\stackrel{\textup{Lem.~\ref{lem:shift}}}{=}&
\g( \,  1^{mc}+ \la'
, \  c^m + \mu, \ m^c+\nu' \, )\\
&=&
\g( \,  (mc) \plp \la
, \  m^c \plp\mu', \ m^c+\nu' \, )
\\
&\stackrel{\textup{Lem.~\ref{lem:shift}}}{=}&
\g( \, c^{l+1}  + ((mc) \plp \la)
, \ 1^e + (m^c \plp\mu'), \ (l+m+1)^c+\nu' \, )
\\
&=&
\g( \, (d) \plp (c^l + \la)
, \ (e) \plp (c^m+\mu), \ c^{l+m+1}\plp\nu \, ).
\end{eqnarray*}

\vspace{-0.6cm}
\end{proof}

\begin{theorem}\label{thm:stable_range}
Let $\la$, $\mu$, $\nu$ be partitions of the same size, such that $\la_1 \geq  \ell(\mu) \cdot \nu_1$ and $\mu_1 \geq \ell(\la) \cdot \nu_1$. Then for every $h \geq 0$ we have
\[
\g(\la,\mu,\nu) = \g(\, \la+h, \ \mu+h, \ \nu+h\,).
\]
\end{theorem}
We provide three proofs of this fact, one in \S\ref{sec:gln}, and two in \S\ref{sec:symfct}. Those sections can be read independently of each other.
The proofs make use of an observation on 3-dimensional contingency arrays with zeros and ones as entries (Lemma~\ref{lem:extremecase}), but they use it in different ways.

We identify subsets $Q \subseteq \IN^3$ with their characteristic functions $Q : \IN^3\to\{0,1\}$, and we call $Q$ a binary or $\{0,1\}$-contingency array.
This means, we interpret $Q$ as a function to $\{0,1\}$, and as the point set of its support. The interpretation will always be clear from the context.
The 2-dimensional marginals of $Q$ are defined as $Q_{i\ast\ast} := \sum_{j,k}Q_{i,j,k} = |Q\cap(\{i\}\times\IN\times\IN)|$,
$Q_{\ast i \ast} := \sum_{j,k}Q_{j,i,k} = |Q\cap(\IN\times\{i\}\times\IN)|$,
$Q_{\ast\ast i} := \sum_{j,k}Q_{j,k,i} = |Q\cap(\IN\times\IN\times\{i\})|$.
For $\alpha\in\IN^\IN$,
$\beta\in\IN^\IN$,
$\gamma\in\IN^\IN$, $|\alpha|=|\beta|=|\gamma|<\infty$,
we denote by
\[
\mathcal{C}(\al,\be,\ga) := \{ Q \subseteq \IN^3 \mid 
Q_{i\ast\ast} = \al_i, \ 
Q_{\ast  i\ast } = \be_i, \
Q_{\ast \ast i} = \ga_i \text{ for every $i$}\}.
\]

\begin{lemma}
\label{lem:upperbound}
For partitions $\al$, $\be$, $\ga$ of equal size, we have
$\g'(\al,\be,\ga)\leq|\mathcal C(\al,\be,\ga)|$.
\end{lemma}
\begin{proof}
There are different proofs of this fact, for example
\cite[Lem.~2.6]{IMW17} and~\cite[Thm.~5.3]{PP20}, see also \S\ref{subsec:toolsgln} and~\S\ref{ss:sym_fun_background}.
\end{proof}

The following lemma shows how restrictions on the marginals can result in strong restrictions on the sets~$Q$, a technique that was also applied in \cite{IMW17}.

\begin{lemma}
\label{lem:extremecase}

Let $\al,\be,\ga$ be compositions with $|\al|=|\be|=|\ga|$. 
Let $a\geq\ell(\alpha)$, $b\geq\ell(\beta)$,
and let the integers $c,h$ be such that  $c+h \geq \ell(\ga)$ and 
$\sum_{i>c} \ga_i \leq h$.
Furthermore, let $\alpha_1\geq bc+h$, $\beta_1\geq ac+h$.

Then,
for every $Q\in\mathcal{C}(\al,\be,\ga)$ we have
\begin{align*}
& \{1\}\times[b]\times[c] \subseteq Q, \quad
[a]\times\{1\}\times[c] \subseteq Q, \quad
\{1\}\times\{1\}\times[c+h] \subseteq Q, \text{ and}\\
& Q \cap (\IN \times \IN \times [c+1,c+h]) = \{1\} \times \{1\} \times [c+1,c+h].
\end{align*}
In particular, if $\mathcal{C}(\al,\be,\ga)$ is non-empty, then
$a=\ell(\al)$, $b=\ell(\be)$,
$\gamma_i=1$ for all $c+1\leq i\leq c+h$,
and $\al_1=bc+h$, $\be_1=ac+h$,
$\al_2\leq bc$,
and $\be_2\leq ac$.

\end{lemma}
\begin{proof}
The situation is depicted in Figure~\ref{fig:threed}.
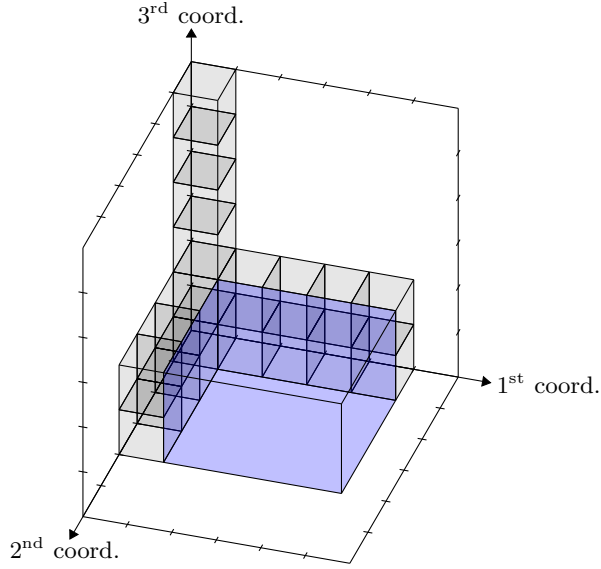
\begin{figure}
\begin{center}
\begin{tikzpicture}[coordinateIIID, scale=1.1]
\grid
\tikzbox{1}{1}{1}
\tikzbox{2}{1}{1}
\tikzbox{3}{1}{1}
\tikzbox{4}{1}{1}
\tikzbox{1}{1}{2}
\tikzbox{2}{1}{2}
\tikzbox{3}{1}{2}
\tikzbox{4}{1}{2}
\tikzbox{1}{2}{1}
\tikzbox{1}{3}{1}
\tikzbox{1}{4}{1}
\tikzbox{1}{5}{1}
\tikzbox{1}{2}{2}
\tikzbox{1}{3}{2}
\tikzbox{1}{4}{2}
\tikzbox{1}{5}{2}
\tikzbox{1}{1}{3}
\tikzbox{1}{1}{4}
\tikzbox{1}{1}{5}
\tikzbox{1}{1}{6}
\tikzcuboid{2}{2}{1}{4}{5}{2}{blue}
\end{tikzpicture}
\end{center}
\caption{Lemma~\ref{lem:extremecase} for $a=5$, $b=4$, $c=2$, $h=4$. A gray cube represents a forced 1 in the contingency array. Absence of color represents a forced 0 in the contingency array. The blue box shows the area where both zeros and ones are possible.}
\label{fig:threed}
\end{figure}
Assume that there exists a binary contingency array $Q\in\mathcal C(\al,\be,\ga)$.
Let
$B_\cup := \{1\}\times[b]\times[c+h] \ \cup \ [a]\times\{1\}\times[c+h]$ be the set of points in the planes  $x=1$ and $y=1$,
and let
$B_\cap := \{1\}\times\{1\}\times[c+h]$ be the set of points on the line $x=y=1$.
Let $H_i := Q \cap (\IN\times\IN\times\{i\}) \cap B_\cup$ be the entries of $Q$ in $B_\cup$ at the section with the plane $z=i$.
In particular, 
$$\sum_{i=1}^{c+h}|H_i|=|Q\cap B_\cup|.$$
We have $\sum_{i=c+1}^{c+h} |H_i|\leq \sum_{i=c+1}^{c+h} \ga_i \leq h$, $|H_i|\leq a+b-1$ for all $0<i\leq c$ and
 $|Q\cap B_\cap| \leq c+h$.
All these inequalities must be met with equality, because
\begin{eqnarray*}
\alpha_1+\beta_1
&=&
|Q\cap B_\cap| + |Q\cap B_\cup|
\\
&=&
\textstyle|Q\cap B_\cap| + \sum_{i=1}^{c+h}|H_i|
\\
&=&
\textstyle|Q\cap B_\cap| + \sum_{i=1}^{c}|H_i| + \sum_{i=c+1}^{c+h}|H_i|
\\
&\leq&
(c+h)+(a+b-1)c + h
\\
&=&
(a+b)c+2h
\\
&\leq&
\alpha_1+\beta_1.
\end{eqnarray*}
We thus have the following equalities:
$|Q\cap B_\cap|=c+h=|B_\cap|$ and $\forall i\in[c]$ we have $|H_i|=a+b-1 = |(\IN\times\IN\times\{i\}) \cap B_\cup|$.
Thus we have $B_\cap\subseteq Q$,
and $\{1\}\times[b]\times[c]\subseteq Q$,
and $[a]\times\{1\}\times[c]\subseteq Q$,
and
$Q \cap (\IN\times\IN\times[c+1,c+h])=\{1\}\times\{1\}\times[c+1,c+h]$.
This gives the desired marginals and the claim follows.
\end{proof}

\begin{proof}[Proof of Theorem~\ref{thm:main}]
Let $\ell(\la) = l$, $\ell(\mu)=m$ and $\nu_1 =c$ and set $d=mc+c$, $e=lc+c$.

Suppose first that $\la_1 \leq mc$ and $\mu_1 \leq lc$. We apply Lemma~\ref{lem:walls}, and obtain
$$\g(\la,\mu,\nu) = \g\big( \underbrace{ (d) \plp (c^l+\la) }_{=:\,\hat{\la}}, \, \underbrace{(e) \plp (c^m + \mu)}_{=:\,\hat{\mu}}, \, \underbrace{c^{l+m+1} \plp \nu}_{=:\,\hat{\nu}} \big).$$

The top rows of $\hat{\la},\hat{\mu},\hat{\nu}$ are $d, e, c$ respectively and thus Theorem~\ref{thm:stable_range} gives that for all $h \in \IN$ we have
\begin{align*}
\g(\hat{\la},\,\hat{\mu},\,\hat{\nu} ) &= \g( \hat{\la}+h,\, \hat{\mu}+h,\, \hat{\nu}+h) \\
 &= 
\g( \, (d+h) \plp (c^l+\la) , \, (e+h) \plp (c^m + \mu), \, (c+h)\plp c^{l+m} \plp \nu \, ) \\
&= \rg( c^l+\la,\, c^m + \mu,\, c^{l+m} \plp \nu ),
\end{align*}
where the last identity follows by letting $h \to \infty$. 
This proves Theorem~\ref{thm:main} in the first case.

Suppose now that $\la_1 > mc$, the case $\mu_1 >lc$ is completely analogous. Set $b:=m+1$.
Then we have $\g(\la,\mu,\nu) = \g(\la',\mu,\nu') =0$ since $\ell(\la') =\la_1>mc = \ell(\mu)\ell(\nu')$. 
On the other hand, the reduced Kronecker coefficient is obtained by adding long first rows, $cm+c+ h$, $cl+c+h$, $c+h$ respectively, so 
\begin{align*}
\rg( c^l+\la, \, c^m + \mu, \, c^{l+m} \plp \nu)
&=
\g\big( (cm +c+ h) \plp (c^l+\la),\, (lc+c+h) \plp (c^m+\mu),\, (c+h)\plp c^{l+m} \plp \nu)\big)
\\
&=
 \g'\big( \underbrace{(cm +c+ h) \plp (c^l+\la)}_{=:\,\al},\, \underbrace{(lc+c+h) \plp (c^m+\mu)}_{=:\,\be},\, \underbrace{( (l+b)^{c} + \nu')\plp(1^h)}_{=:\,\ga}\big)
\end{align*}
for sufficiently large $h$.
Let $\hat\ga = (l+b)^c+\nu'$ be $\gamma$ without the $h$ many trailing 1s.
We observe that
$\al_2 = \la_1+c$, $\ell(\be)=b$, and $\ell(\hat\ga)=c$.
From $\la_1 > mc$
we conclude
$\al_2 >bc$.
Lemma~\ref{lem:extremecase} shows that $\mathcal C(\al,\be,\ga)=\emptyset$.
Hence $\g'(\al,\be,\ga)=0$ by Lemma~\ref{lem:upperbound}.
\end{proof}

\section{Proofs via the general linear group}
\label{sec:gln}

\subsection{Tools from the general linear group viewpoint}
\label{subsec:toolsgln}
We refer to \cite[\S8]{Ful97} for the basic properties of the irreducible representations of the general linear group.
The irreducible representations $V_{a^b}(\IC^b)$ of the general linear group $\GL_b$ are 1-dimensional: For $g\in\GL_b$, $v\in V_{a^b}(\IC^b)$, we have $gv := \det(g)^a v$.
Hence, if we decompose $V_{1^{ab}}(\IC^{ab})$ as a $\GL_a\times\GL_b$ representation via the group homomorphism $\GL_a\times\GL_b\to\GL_{ab}$, $(g_1,g_2)\mapsto g_1\otimes g_2$, then we obtain
$V_{1^{ab}}(\IC^{ab})\simeq V_{b^a}(\IC^b)\otimes V_{a^b}(\IC^b)$.
Tensoring with such a 1-dimensional representation preserves irreducibility: $V_\la(\IC^a)\otimes V_{b^a}(\IC^a) \simeq V_{b^a+\la}(\IC^a)$.

The Kronecker coefficients have an interpretation
as the structure coefficients arising when decomposing irreducible $\GL_{ab}$ representations as $\GL_a\times\GL_b$ representations:
\[
V_\nu(\IC^{ab}) \simeq \bigoplus_{\substack{\la \vdash_a |\nu| \\ \mu \vdash_b |\nu|}} \left(
V_\la(\IC^{a}) \otimes V_\mu(\IC^{b})
\right)^{\oplus\g(\la,\mu,\nu)}.
\]
This can be seen directly from Schur-Weyl duality, see e.g.~\cite[(2.2)]{Chr06} or \cite[Pro.~4.4.11]{Ike12}.

Another formulation is via the multiplicity of the irreducible $G := \GL_a\times\GL_b\times\GL_c$ representation $V_\alpha(\IC^a)\otimes V_\beta(\IC^b)\otimes V_\gamma(\IC^c)$ in the $D$-th wedge power of $\IC^{a}\otimes\IC^{b}\otimes\IC^{c}$, see \cite{IMW17}. Formally for partitions $\al,\be,\ga\vdash D$ we have
\[
\g'(\al,\be,\ga) \ = \ \textup{mult}_{\al,\be,\ga}\left(
\wedgepower^D(\IC^{a}\otimes\IC^{b}\otimes\IC^{c})
\right).
\]
A vector $v$ for which $\big(\diag(r_1,\ldots,r_a),\diag(s_1,\ldots,s_b),\diag(t_1,\ldots,t_c)\big)v = r_1^{\la_1}\cdots r_a^{\la_a}\cdot
s_1^{\mu_1}\cdots s_b^{\mu_b}\cdot
t_1^{\nu_1}\cdots t_c^{\nu_c} v$
is called a \emph{weight vector} of weight $(\la,\mu,\nu)$.

For $(A,B,C)\in \IC^{a\times a} \times \IC^{b\times b} \times \IC^{c\times c}$,
the Lie algebra action on $\wedgepower^D(\IC^a\otimes\IC^b\otimes\IC^c)$
is defined as $(A,B,C).v := \lim_{\varepsilon\to0}\varepsilon^{-1}((\varepsilon (A,B,C) + (\textup{id}_a,\textup{id}_b,\textup{id}_c))v-v)$.
A raising operator is the Lie algebra action of $(E_{i-1,i},0,0)$, where $E_{i,j}$ is the matrix with a 1 at position $(i,j)$ and zeros everywhere else.
 The other raising operators are $(0,E_{i-1,i},0)$ and $(0,0,E_{i-1,i})$.
Let $e_i := (0,\ldots,0,1,0,\ldots,0)^T$
and let $e_{i,j,k} := e_i \otimes e_j \otimes e_k$.
Then, for example, $(E_{i,j},0,0) e_{r,1,1} = e_{i,1,1}$ iff $r=j$ and 0 otherwise.
A highest weight vector (HWV) of weight $(\alpha,\beta,\gamma)$ is a nonzero weight vector of weight $(\alpha,\beta,\gamma)$ that is mapped to zero by all raising operators.
The
irreducible $\GL_a\times\GL_b\times\GL_c$
representation
$V_{\al}\otimes V_{\be}\otimes V_{\ga}$
contains exactly one HWV (up to scale), and that is of weight $(\al,\be,\ga)$. Hence (\cite[Lemma~2.1]{IMW17}),
\[
\g'(\alpha,\beta,\gamma) = \dim\left(\HWV_{\alpha,\beta,\gamma}\wedgepower^D(\IC^{a}\otimes\IC^{b}\otimes\IC^{c})\right),
\]
where $\HWV_{\alpha,\beta,\gamma}$ denotes the space of HWVs of weight $(\alpha,\beta,\gamma)$.
Note that
each standard basis vector
in $\wedgepower^D(\IC^{a}\otimes\IC^{b}\otimes\IC^c)$
is a weight vector, and hence for each weight vector space of weight $w$ we have a basis given by the set of standard basis vectors of weight $w$.
Let $e_{i,j,k} := e_i \otimes e_j \otimes e_k$,
and for a list of points $Q \in (\IN^3)^D$ we define $\psi_Q := e_{Q_1} \wedge e_{Q_2} \wedge \cdots \wedge e_{Q_D}$.
If $Q$ has marginals $(\al,\be,\ga)$, then $\psi_Q$ has weight $(\al,\be,\ga)$.
This immediately implies the result of Lemma~\ref{lem:upperbound}.

We illustrate the concept with some examples.
The HWVs in $\bigwedge^2(\IC^2\otimes\IC^2\otimes\IC^1)$ are
$e_{1,1,1}\wedge e_{2,1,1}$ and 
$e_{1,1,1}\wedge e_{1,2,1}$.
A nontrivial example is the HWV
$$
t :=
e_{1,1,1}\wedge e_{2,1,1} \wedge e_{1,2,2}
+
e_{1,1,1}\wedge e_{1,2,1} \wedge e_{2,1,2}
+
e_{1,1,1}\wedge e_{1,1,2} \wedge e_{2,2,1}
$$
of weight $((2,1),(2,1),(2,1))$ in $\bigwedge^3(\IC^2\otimes\IC^2\otimes\IC^2)$,
which can be seen as follows:
\begin{align*}
(E_{1,2},0,0)t &=
e_{1,1,1}\wedge e_{1,2,1} \wedge e_{1,1,2}
+
e_{1,1,1}\wedge e_{1,1,2} \wedge e_{1,2,1}
=0,
\\
(0,E_{1,2},0)t &=
e_{1,1,1}\wedge e_{2,1,1} \wedge e_{1,1,2}
+
e_{1,1,1}\wedge e_{1,1,2} \wedge e_{2,1,1}
=0
, \\
(0,0,E_{1,2})t &=
e_{1,1,1}\wedge e_{2,1,1} \wedge e_{1,2,1}
+
e_{1,1,1}\wedge e_{1,2,1} \wedge e_{2,1,1}
=0
.
\end{align*}

\subsection{Proofs from the general linear group viewpoint}
For the sake of completeness, we present the short proof of Lemma~\ref{lem:shift} from \cite[Pf of Lem~2.1]{BOR09}
 and \cite[Cor.~4.4.14]{Ike12}.
\begin{proof}[Proof of Lemma~\ref{lem:shift} via the general linear group]
We have that $V_{m^l+\la}(\IC^l)\otimes V_{l^m+\mu}(\IC^m)$ occurs in $V_{1^{lm}+\nu}(\IC^{lm})$ with multiplicity $\g(\,m^l+\la,\ l^m+\mu,\ 1^{lm}+\nu\,)$.
But we also have
\begin{eqnarray*}
V_{1^{lm}+\nu}(\IC^{lm})
&\simeq&
V_{1^{lm}}(\IC^{lm}) \otimes V_{\nu}(\IC^{lm})
\\
&\simeq&
(V_{m^l}(\IC^l)\otimes V_{l^m}(\IC^m)) \ \otimes \
\bigoplus_{\substack{\la \vdash_l |\nu| \\ \mu \vdash_m |\nu|}} \left(
V_\la(\IC^{l}) \otimes V_\mu(\IC^{m})
\right)^{\oplus\g(\la,\mu,\nu)}
\\
&\simeq&
\bigoplus_{\substack{\la \vdash_a |\nu| \\ \mu \vdash_b |\nu|}} \left(
V_{m^l+\la}(\IC^{l}) \otimes V_{l^m+\mu}(\IC^{m})
\right)^{\oplus\g(\la,\mu,\nu)}.
\end{eqnarray*}
\mbox{~}

\vspace{-1cm}
\end{proof}

\begin{proof}[Proof of Theorem~\ref{thm:stable_range} via contingeny arrays and highest weight vectors]
Let $a:=\ell(\la)$, $b:=\ell(\mu)$, $c:=\nu_1$.
Let $\gamma := \nu'$, so $\ell(\gamma)=c$.
We have $\la_1 \geq bc$ and $\mu_1 \geq ac$.
Observe that
$\g(\la,\mu,\nu) = \g'(\la,\mu,\gamma)$.
Let
$\widetilde \la = \la + (h)$,
$\widetilde \mu = \mu + (h)$,
$\widetilde \gamma = \gamma \plp (1^h)$.
We define an injective linear map $\varphi$ as follows.
\begin{eqnarray*}
\varphi : \  \wedgepower^D(\IC^{a}\otimes\IC^{b}\otimes\IC^{c})
&\to&
\wedgepower^{D+h}(\IC^{a}\otimes\IC^{b}\otimes\IC^{c+h})
\\
v &\mapsto& v \wedge e_{1,1,c+1} \wedge e_{1,1,c+2} \wedge \cdots \wedge e_{1,1,c+h}
\end{eqnarray*}
Note that $\varphi$ maps vectors of weight $(\la,\mu,\ga)$ to vectors of weight $(\widetilde\la,\widetilde\mu,\widetilde\ga)$.
It remains to show that $\varphi$ maps HWVs to HWVs, and that every HWV of weight $(\widetilde\la,\widetilde\mu,\widetilde\ga)$ has a preimage under $\varphi$.

We first prove that $\varphi$ sends HWVs to HWVs.
By construction of $\varphi$, we observe that for $1 \leq i < i' \leq a$, we have
$$(E_{i,i'},0,0) \varphi(u) = \varphi((E_{i,i'},0,0)u) = \varphi(0) = 0.$$ Analogously, $(0,E_{j,j'},0) \varphi(u) = 0$ for $1 \leq j < j' \leq b$, and $(0,0,E_{k,k'}) \varphi(u) = 0$ for $1 \leq k < k' \leq c$. 
The remaining raising operators also vanish by construction of $\varphi$, because 
$$(0,0,E_{k,k'}) (v \wedge e_{1,1,c+1} \wedge \cdots \wedge e_{1,1,c+h}) = v \wedge e_{1,1,c+1} \wedge \cdots \wedge \widehat{e_{1,1,c+k}} \wedge e_{1,1,c+k'} \wedge e_{1,1,c+k'} \wedge \cdots \wedge e_{1,1,c+h} = 0$$
 because of the repeated factor $e_{1,1,c+k'}$. Here the $\widehat{e_{1,1,c+k}}$ means omission of that factor.

We now show that every weight vector of weight $(\widetilde\la,\widetilde\mu,\widetilde\ga)$ has a preimage under $\varphi$, which finishes the proof.
It is sufficient to show this for basis vectors.
Let $u = \psi_P$ be a basis weight vector of weight $(\widetilde\la,\widetilde\mu,\widetilde\ga)$,
i.e., $Q\subseteq\IN^3$ with marginals $(\widetilde\la,\widetilde\mu,\widetilde\ga)$.
We apply Lemma~\ref{lem:extremecase} to see that $\{1\}\times\{1\}\times[c+1,c+h]\subset Q$
and $Q\cap(\IN\times\IN\times\{i\})=\{(1,1,i)\}$ for all $c+1\leq i\leq c+h$.
Therefore, $\psi_Q$ has a preimage under $\varphi$, namely $\psi_{P}$, where $P$ arises from $Q$ by deleting all points with 3rd coordinate $>c$.
\end{proof}

\section{Proofs via symmetric functions}
\label{sec:symfct}

\subsection{Tools from symmetric functions}
\label{ss:sym_fun_background}

Here we recall basic definitions and facts from symmetric function theory, see~\cite{EC2, Sag, Mac}. 

The standard Young tableaux (SYT) of shape $\la \vdash n$ are assignments of $1,2,\ldots,n$ to the Young diagram of $\la$, so that the numbers are decreasing along rows and down columns, and each number appears exactly once. A semi-standard Young tableaux (SSYT) of skew shape $\la/\mu$ is an assignment of integers $1,2,\ldots,N$ to the boxes of the skew Young diagram $\la/\mu$, such that the values weakly increase along rows and strictly down columns. We say that an SSYT $T$ has type $type(T)=\al$ if there are $\al_i$ entries equal to $i$ for each $i$. 

\ytableausetup{boxsize=2ex}
As an example $\ytableaushort{\none\none 113,\none233,1344}$ is an SSYT of shape $(5,4,4)/(2,1)$ and type $(3,1,4,2)$.

The ring of symmetric functions $\Lambda$ has several fundamental bases. Here we will use the monomial basis $\{ m_\la\}$ given by $m_\la =x_1^{\la_1}x_2^{\la_2}\ldots +\cdots$, summing over all distinct monomials with exponents $\la_1,\la_2,\ldots$. We will also use the homogeneous symmetric functions $\{h_\la\}$ given by
$$h_m := \sum_{i_1\leq i_2 \leq \cdots \leq i_m} x_{i_1} x_{i_2}\cdots x_{i_m} \text{ for }m>0; \quad h_0=1; \quad h_{m}=0 \text{ for }m<0; \text{ and }\quad h_\la := h_{\la_1} h_{\la_2}\cdots .$$
The \emph{Schur functions} $s_\la$ are one of the fundamental bases of the ring $\Lambda$ of symmetric functions. Moreover, $s_\la(x_1,\ldots,x_N)$ is the value of the character of $V_\la$ at a matrix with eigenvalues $x_1,\ldots,x_N$. 

We have the following formulas for them, where $\ell=\ell(\la)$,
\begin{align}
\text{Jacobi-Trudi identity:} \quad & s_\la = \det[ h_{\la_i -i+ j} ]_{i,j=1}^\ell \\
\text{Weyl determinantal formula:} \quad & s_\la(x_1,\ldots,x_N) = \frac{ \det[ x_i^{\la_j+N-j}]_{i,j=1}^N }{\det[x_i^{N-j}] } \label{eq:weyl}\\
\text{via SSYTs:} \quad & s_\la = \sum_{T \in SSYT(\la) } x^{type(T)}
\end{align}

The Littlewood-Richardson coefficients are the structure constants in the ring of symmetric functions as
$$s_\mu(x) s_\nu(x) = \sum_\la c^\la_{\mu,\nu} s_\la(x).$$
They can be computed combinatorially via the Littlewood-Richardson rule: 
$c^{\la}_{\mu\nu}$ is equal to the number of SSYTs $T$ of shape $\la/\mu$, type $\nu$, and whose reading words is a ballot sequence. The  reading word is obtained by reading the tableaux right to left along rows, top to bottom, and a word is a ballot sequence if in every prefix the number of $i$'s is not less than the number of $i+1$'s for every $i$. 
The multi-LR coefficients $c^{\la}_{\al^1 \cdots \al^k}$ are defined as
\begin{equation}\label{eq:multi-LR}
c^{\la}_{\al^1\cdots \al^k} := \langle s_\la, s_{\al_1} s_{\al^2}\cdots s_{\al^k} \rangle = \sum_{\be^1,\be^2,\ldots} c^{\la}_{\al^1 \be^1} c^{\be^1}_{\al^2 \be^2} \cdots c^{\be^{k-1}}_{\al^{k-1}\al^k}
\end{equation}
from where it is easy to see that they count SSYTs $T$ of shape $\la$ and type $(\al^1 \plp \al^2 \plp \cdots)$, such that the reading word of each skew subtableau corresponding to the entries with values between $1+\sum_{i=1}^r \ell(\al^i)$  and $\sum_{i=1}^{r+1} \ell(\al^i)$ is a lattice permutation for every $r=1,\ldots,k-1$. For example,
\ytableausetup{boxsize=2.5ex}
$$\ytableaushort{1111446,222457,35566} *[*(red!30)]{4,3,1}*[*(blue!30)]{4+2,3+2,1+2}*[*(yellow!30)]{6+1,5+1,3+2} \qquad \text{ and } \qquad \ytableaushort{1111446,222466,35557} *[*(red!30)]{4,3,1}*[*(blue!30)]{4+2,3+1,1+3}*[*(yellow!30)]{6+1,4+2,4+1}$$
are two  multi-LR tableaux of shape $\la =(7,6,5)$ and types $\al^1=(4,3,1)$, $\al^2 =(3,3)$, $\al^3=(3,1)$.

\medskip

The Kronecker coefficient can be studied via the following expansions
\begin{align}\label{eq:schur_kron1}
s_\la[x\cdot y] = \sum_{\mu,\nu} \g(\la,\mu,\nu) s_\mu(x) s_\nu(y),
\end{align}
where $x \cdot y = (x_1y_1,x_1y_2,\ldots,x_2y_1,\ldots)$ consists of the pairwise products of the two sets of variables.
In particular, this gives that 
$$h_m[x \cdot y] = \sum_\la s_\la(x)s_\la(y).$$
From the Jacobi-Trudy identity we thus obtain 
$$s_\la[x\cdot y] = \det[ h_{\la_i-i+j}[x\cdot y] ] = \sum_{\sigma \in S_{\ell}} \sgn(\sigma) \sum_{\, \al^i \vdash \la_i-i+\sigma_i} s_{\al^1}(x)\cdots s_{\al^\ell}(x)s_{\al^1}(y)\cdots s_{\al^\ell}(y),$$
so
\begin{align}\label{eq:kron_multilr}
\g(\la,\mu,\nu) = \sum_{\sigma \in S_{\ell}} \sgn(\sigma) \sum_{\, \al^i \vdash \la_i-i+\sigma_i} c^{\mu}_{\al^1\cdots \al^k}c^{\nu}_{\al^1\cdots \al^k}.
\end{align}
Note that this identity appears many times in the literature, including~\cite{Val09,PP17,PPq}.

The following are referred as the ``triple Cauchy identities'', see e.g. \cite[Exercise 7.78]{EC2}:
$$\sum_{\la,\mu,\nu} \g(\la,\mu,\nu) s_\la(x) s_\mu(y) s_\nu(z) = \prod_{i,j,k} \frac{1}{1-x_iy_jz_k},$$ 
$$\sum_{\la,\mu,\nu} \g(\la,\mu,\nu) s_\la(x) s_\mu(y) s_{\nu'}(z) = \prod_{i,j,k} (1+x_iy_jz_k),$$ 
where the second identity follows from the first via the involution $\omega$ on the symmetric functions in the variables $z$.
Denote by $C(\al,\be,\ga):=|\mathcal{C}(\al,\be,\ga)|$. Then the second identity becomes
\begin{align}\label{eq:kron_ct_0}
\sum_{\la,\mu,\nu} \g(\la,\mu,\nu) s_\la(x) s_\mu(y) s_{\nu'}(z) = \sum_{\al,\be,\ga} C(\al,\be,\ga) x^\al y^\be z^\ga
\end{align}
Note that this identity immediately gives the upper bound in Lemma~\ref{lem:upperbound} by comparing coefficients at $x^\la y^\mu z^{\nu'}$ on both sides.

We now express $\g(\la,\mu,\nu)$ as an alternating sum over contingnecy arrays. Denote by 
$$\Delta(x) = \det [x_i^{a-j}] = \prod_{i<j}(x_i-x_j) = \sum_{\sigma \in S_a} \sgn(\sigma) x_1^{a-\sigma_1}\cdots x_a^{a - \sigma_a},$$
and multiply by $\Delta(x) \Delta(y) \Delta(z)$ both sides of equation~\eqref{eq:kron_ct_0}. Expressing the Schur functions via the ratio of determinants in~\ref{eq:weyl}, we obtain
$$\sum_{\la,\mu,\nu} \g(\la,\mu,\nu) \det[x_i^{\la_j+a-j}]\det[y_i^{\mu_j+b-j}] \det[z_i^{\nu'_j +c-j}] = \Delta(x)\Delta(y)\Delta(z) \sum_{\al,\be,\ga} C(\al,\be,\ga) x^\al y^\be z^\ga$$ 
Comparing coefficients at $x_1^{\la_1+a-1}\ldots y_1^{\mu_1+b-1}\ldots z_1^{\nu'_1 + c-1} \ldots$ on both sides
we obtain 
$$\g(\la,\mu,\nu) = [x_1^{\la_1+a-1}\ldots y_1^{\mu_1+b-1}\ldots z_1^{\nu'_1 + c-1} \ldots]\Delta(x)\Delta(y)\Delta(z) \sum_{\al,\be,\ga} C(\al,\be,\ga) x^\al y^\be z^\ga, $$
where the $[\cdots]$ denotes the coefficient extraction. 
Expanding the $\Delta$s into monomials, whose marginals we incorporate, we get that we must have
$\la_i + a - i = \al_i + a-\sigma_i$ etc, so $\al_i = \la_i+\sigma_i-i$ and we obtain, see also~\cite{PP20},
\begin{align}\label{eq:kron_ct}
\g(\la,\mu,\nu) = \sum_{\sigma \in S_a , \, \pi \in S_b, \,  \rho \in S_c} \sgn(\sigma)\sgn(\pi)\sgn(\rho) C(\la + \sigma-\id, \mu + \pi -\id, \nu' +\rho -\id). 
\end{align}
where a permutation $\sigma$ is interpreted as the vector $(\sigma(1),\ldots,\sigma(a))$ and $\id=(1,2,\ldots)$ is the identity permutation of the corresponding size.

\subsection{Proofs via symmetric functions}\label{ss:sym_fun_proofs}

\begin{proof}[Proof of Lemma~\ref{lem:shift} via symmetric functions]

Let $\hat{\nu}=1^{l m}+ \nu$.
We use Schur functions as follows. We  apply equation~\eqref{eq:schur_kron1} with variables $x_1,\ldots,x_\ell$ and $y_1,\ldots,y_m$, so $x_i=0$ for $i>l$ and $y_j=0$ for $j>m$, and we obtain
\begin{align}\label{eq:schur_1}
s_{\hat{\nu}}[x\cdot y] = \sum_{\theta, \tau} \g(\hat{\nu},\theta,\tau) s_{\theta}(x) s_{\tau} (y).
\end{align}
If $\g(\hat{\nu},\theta,\tau) >0$, we must have $\ell(\theta) \ell(\tau) \geq \ell(\hat{\nu}) = l m$. Since $s_\theta(x_1,\ldots,x_l)=0$ if $\ell(\theta) >l$ and $s_\tau(y_1,\ldots,y_m)=0$ if $\ell(\tau)>m$, we then must have only the partitions with $\ell(\theta)=l$, $\ell(\tau)=m$ appearing.

Since $s_{\hat{\nu}}$ is the generating function over SSYTs with entries $x_1y_1,\ldots,x_ly_m$, and its first column has length exactly $lc$, we must have all the entries $x_iy_j$ appearing exactly once in that column. As this is the minimal possible column, the rest of the SSYT can be any of the SSYTs of the remaining shape and entries $x_1y_1,\ldots,x_ly_m$. %Let $\rho = (\mu+'(b-1)^c)$, so $\be = (ac,\rho')'$.  
Thus 
\begin{align}\label{eq:schur_2}
s_{\hat{\nu}}[x\cdot y] = s_\nu[x\cdot y] \prod_{i,j} x_iy_j  = (x_1\ldots x_l)^m(y_1,\ldots,y_m)^l \sum_{\rho,\eta} \g(\nu,\rho,\eta) s_\rho(x) s_\eta(y)
\end{align} 
We also note that $s_{l^m + \mu}(y_1,\ldots,y_m) = (y_1\ldots y_m)^l s_\mu(y)$, since the first $l$ columns of length $m$ are forced to be filled with $1,\ldots,m$, and for the remaining tableaux there are no restrictions other than being an SSYT. Similarly $s_{m^l + \la}(x_1,\ldots,x_l) = (x_1\ldots x_l)^m s_\la(x)$. Comparing coefficients of $s_\la(x)s_\mu(y)$ at equations \eqref{eq:schur_1} and \eqref{eq:schur_2} we thus see that 
$$\g(\hat{\nu}, l^m + \mu, m^l + \la) = \g(\nu,\la,\mu).$$

\mbox{~}\vspace{-1cm}

\end{proof}

\begin{proof}[Proof of Theorem~\ref{thm:stable_range} via contingency arrays and symmetric functions]
From now on we will use  formula~\eqref{eq:kron_ct} and Lemma~\ref{lem:extremecase} to show that the only possible contingency arrays are the ones in Figure~\ref{fig:lem:walls}.

Consider now $\g(\la + h, \mu+h, \nu+h)$ as in the problem, and let $\al = (\la+h), \be = (\mu+h)$, $\ga=(\nu+h)'$  so that 
$\g(\al,\be,\ga') = \g(\la + h, \mu+h, \nu+h)$. Let $\nu_1=c$, $\ell(\la)=a$ and $\ell(\mu)=b$, so we have
 $\al_1 \geq bc+h, \be_1 \geq ac+h$, $\ga_i =1$ for $i=c+1,\ldots,c+h$ and
\begin{align}\label{eq:kron_ct_h}
\g(\al,\be,\ga') = \sum_{\sigma \in S_a,\, \pi \in S_b,\, \rho \in S_{c+h}} \sgn(\sigma)\sgn(\pi)\sgn(\rho) C(\al + \sigma-\id, \be + \pi -\id, \ga +\rho -\id)
\end{align}

In formula~\eqref{eq:kron_ct_h} we then consider $\{0,1\}$-contingency arrays $Q$ with
marginals 
$$Q_{1\ast\ast}:=\sum_{j,k} Q_{1,j,k} = \la_1 +\sigma_1-1 \geq bc+h, \quad Q_{\ast 1 \ast}:= \sum_{i,k} Q_{i,1,k} = \mu_1+ \pi_1-1 \geq ac+h,$$
$$Q_{\ast \ast k}:= \sum_{i,j} Q_{i,j,k} = 1+ \rho_k - k, \text{ for }  k=c+1,\ldots,c+h.$$
 Note that then we have 
$$\sum_{k>c} Q_{\ast\ast k} = h + \sum_{k=c+1}^{c+h} \rho_k - \sum_{k=c+1}^{c+h} k \leq h,$$

and the support of the array is in $[1,a]\times [1,b] \times [1,c+h]$,
so we can apply Lemma~\ref{lem:extremecase} and conclude that 
$$Q_{1,j,k}=0 \text{ iff } (j,k) \in [2,b]\times [c+1,c+h], \quad Q_{i,1,k}=0 \text{ iff } (i,k) \in [2,a] \times [c+1,c+h].$$

Thus, we must have $Q_{1\ast\ast}=bc+h$, $Q_{\ast 1 \ast} =ac+h$  and so $\sigma_1=\pi_1=1$, $\{\rho_{c+1},\ldots,\rho_{c+h}\}=\{c+1,\ldots,c+h\}$ and for $k \in [c+1,c+h]$ we must have $Q_{i,j,k}=0$ unless $i=j=1$. This also forces us to have $Q_{1,1,k}=1$ for all these $k$, and so $\rho_k =k$ for $k=c+1,\ldots,c+h$.  

This completely determines $Q_{i,j,k}$ for $k>c$, as well as $\rho_k$ for $k>c$, and $\rho = \bar{\rho},(c+1),\ldots,(c+h)$ for $\bar{\rho}\in S_c$. We can thus write formula~\eqref{eq:kron_ct_h} as
\begin{align*}
\g(\la+h,\mu+h,\nu+h) = \sum_{\sigma \in S_a, \,\pi \in S_b,\,\rho \in S_{c+h}} \sgn(\sigma)\sgn(\pi)\sgn(\rho) C(\al + \sigma-id, \be + \pi -id, \ga +\rho -id)\\
= \sum_{\sigma \in S_a,\,\pi \in S_b,\,\eta \in S_{c}} \sgn(\sigma)\sgn(\pi)\sgn(\eta) C(\bar{\al} + \sigma-id, \bar{\be} + \pi -id, \bar{\ga} +\eta -id),
\end{align*}
Where $\bar{\al} = \al - (h) = \la$, $\bar{\be} = \be - (h) = \mu$ and $\bar{\ga} = (\ga_1\ldots,\ga_c) = \nu'$. 
As the last part coincides with the expression for $\g(\la,\mu,\nu)$ in~\eqref{eq:kron_ct}, we get the desired identity. 
\end{proof}

\begin{proof}[Proof of Theorem~\ref{thm:stable_range} via Littlewood-Richardson coefficients]
Let again $\ell(\la)=a, \ell(\mu)=b$ and $\nu_1=c$.

We have that $\g(\la+h,\mu+h,\nu+h) = \g( \nu'\plp (1^h), \la' \plp (1^h), \mu+h)$ and we are going to apply formula~\eqref{eq:kron_multilr} with that triple of partitions. Set $\hat{\mu}=\mu+h$, $\hat{\la} = \la' \plp (1^h) = (\la+h)'$ and $\hat{\nu}=\nu'\plp(1^h)(\nu+h)'$.  Here $\ell(\nu'\plp(1^h)) = c+h$, so
$$\g(\la+h, \mu+h,\nu+h) = \sum_{\sigma \in S_{c+h} } \sgn(\sigma) \sum_{\al^i \vdash \hat{\nu}_i-i+\sigma_i} c^{\hat{\la}}_{\al^1 \al^2 \cdots} c^{\hat{\mu}}_{\al^1 \al^2 \cdots}$$

We will now characterize the possible partitions $\al^i$ involved in this sum. From the iterated definition of the multi-LR coefficients~\eqref{eq:multi-LR} we see that in order for the coefficients to be nonzero, we must have $\al^i \subset \hat{\mu}$ and $\al^i \subset \hat{\la}$. In particular then $\ell(\al^i) \leq \ell(\mu) =b$ and $\al^i_1 \leq \hat{\la}_1 = a$. Note that multi-LR coefficients count certain SSYTs of type $(\al^1\plp \al^2 \plp\ldots \plp \al^c \plp \ldots)$ and thus in the shape $\hat{\la}$ the first column will have at most $\ell(\al^1) + \cdots +\ell(\al^c) \leq bc$ many entries from the first $c$ partitions. So there are at least $h$ boxes in the first column which need to be covered by the partitions $\al^{c+1},\ldots, \al^{c+h}$. We then have 
$$h \leq \ell(\al^{c+1})+\cdots+\ell(\al^{c+h}) \leq |\al^{c+1}| + \cdots + |\al^{c+h}| = \sum_{i=c+1}^{c+h} 1 -i +\sigma_i \leq h,$$
as $\sigma_{c+1}+\cdots+\sigma_{c+h} \leq c+1+\cdots c+h$. Thus we need to have equalities, and so

$$|\al^{c+1}| + \cdots + |\al^{c+h}| =h, \ell(\al^i)=|\al^i|,$$
so $\al^i$ are single column partitions, possibly empty.  

Further, we have $\al^i \leq a$,
$\al^i \subset \hat{\mu}$. As there is a multi-LR of type $(\al^1 \plp \al^2 \cdots)$, the first row of that tableaux can only be occupied by the smallest entries of each type. So  we must have 
$$ac+h =\hat{\mu}_1 \leq \sum_i \al^i_1 \leq \sum_{i=1}^c a + \sum_{i=c+1}^{c+h} \al^i_1.$$  
Thus $\al_1^{c+1}+\cdots + \al^{c+h}_1 \geq h$. Since $\al^i_1\leq 1$ by the above consideration, we must have $\al^i=(1)$ for all $i>c$.  So $\sigma_i=i$ for $i=c+1,\ldots,c+h$. 

Then 
$$c^{\hat{\la}}_{\al^1 \al^2 \cdots \al^{c+h}} = c^{\la'}_{\al^1\cdots \al^c} \quad \text{and} \quad  c^{\hat{\mu}}_{\al^1 \al^2 \cdots \al^{c+h} } = c^{\mu}_{\al^1\cdots \al^c}.$$
We thus get that 
\begin{align*}
\g(\la+h, \mu+h,\nu+h) = \sum_{\sigma \in S_{c+h} } \sgn(\sigma) \sum_{\al^i \vdash \hat{\nu}_i-i+\sigma_i} c^{\hat{\la}}_{\al^1 \al^2 \cdots} c^{\hat{\mu}}_{\al^1 \al^2 \cdots}\\
=\sum_{\sigma \in S_{c} } \sgn(\sigma) \sum_{\al^i \vdash \nu'_i-i+\sigma_i} c^{\la'}_{\al^1 \al^2 \cdots} c^{\mu}_{\al^1 \al^2 \cdots} =\g(\nu',\la',\mu)=\g(\la,\mu,\nu),
\end{align*}
which completes the proof.
\end{proof}

We now discuss how Theorem~\ref{thm:main} can be seen from the following identity, which first appeared in~\cite{BOR2}, where it was proven using symmetric function operators. It was then reformulated in
~\cite[Theorem 4.3]{BDO}, which studies the partition algebra, as follows.

Set $m = r + s$ and let $\nu\vdash m-l$, $\la\vdash r$, $\mu\vdash s$ for some non-negative integer $l$.
Then
\[
\rg(\la,\mu,\nu) =
\sum_{\substack{l_1, l_2\\l=l_1+2l_2}}
\sum_{\substack{\alpha\vdash r-l_1-l_2\\\beta\vdash s-l_1-l_2}}
\sum_{\substack{\pi,\rho,\sigma\vdash l_1\\\gamma\vdash l_2}}
c_{\al,\be,\pi}^\nu
c_{\al,\rho,\ga}^\la
c_{\ga,\sigma,\be}^\mu
\g(\pi,\rho,\sigma)
\]

We now apply it to compute the reduced Kronecker in Theorem~\ref{thm:main}.
Let $\hat{\la} = \nu_1^{\ell(\la)}+\la$, $\hat{\mu} = \nu_1^{\ell(\mu)}+\mu$ and $\hat{\nu} = (\nu_1^{\ell(\la)+\ell(\mu)},\nu)$. Then $m -l  = (\ell(\la)+\ell(\mu))\nu_1 + n$, $r= \ell(\la)\nu_1+n$, $s =\ell(\mu)\nu_1+n$, so $l=n$ and  the above identity translates to
$$
\rg(\hat{\la},\hat{\mu},\hat{\nu}) = \sum_{\substack{l_1, l_2\\n=l_1+2l_2}}
\sum_{\substack{\alpha\vdash r-l_1-l_2\\\beta\vdash s-l_1-l_2}}
\sum_{\substack{\pi,\rho,\sigma\vdash l_1\\\gamma\vdash l_2}} c_{\al,\be,\pi}^{\hat{\nu}}
c_{\al,\rho,\ga}^{\hat{\la}}
c_{\ga,\sigma,\be}^{\hat{\mu}}
\g(\pi,\rho,\sigma)
$$

We now observe that $|\al| \geq \ell(\la)\nu_1$, and if $c^{\hat{\la}}_{\al,\rho,\ga}>0$, $c^{\hat{\nu}}_{\al,\be,\pi}>0$ then $\al \subset \hat{\nu} \cap \hat{\la} = (\nu_1^{\ell(\la)})$. Then we must have $\al = \nu_1^{\ell(\la)}$ and so $l_1+l_2=n$. Since $l_1 +2l_2=n$, we must have $l_2=0$ and $l_1=n$, so $\gamma=\emptyset$. Similarly, we obtain $\be=(\nu_1^{\ell(\mu)})$, leaving us with
$$
\rg(\hat{\la},\hat{\mu},\hat{\nu}) = 
\sum_{\pi,\rho,\sigma\vdash n} c_{\al,\be,\pi}^{\hat{\nu}}
c_{\al,\rho}^{\hat{\la}}
c_{\sigma,\be}^{\hat{\mu}}
\g(\pi,\rho,\sigma),
$$
Observe that the Littlewood-Richardson rule gives, since $\al$ is the rectangle in the beginning of $\hat{\la}$, that $c^{\hat{\la}}_{\al,\rho} = 0$ if $\rho \neq \hat{\la} /\al = \la$, and is $1$ otherwise. Similarly the other Littlewood-Richardson coefficients are zero unless $\sigma=\mu$ and $\pi=\nu$, we are left with only one partition triple   $(\pi,\rho,\sigma) = (\nu,\la,\mu)$, whose coefficient is 1 and the identity in Theorem~\ref{thm:main} follows.


{\footnotesize
\begin{thebibliography}{BDVO15}

\bibitem[BB17]{BB17}
Christine Bessenrodt and Christopher Bowman.
\newblock Multiplicity-free {K}ronecker products of characters of the symmetric
  groups.
\newblock {\em Advances in Mathematics}, 322:473--529, 2017.

\bibitem[BDVO15]{BDO}
Christopher Bowman, Maud De~Visscher, and Rosa Orellana.
\newblock The partition algebra and the {K}ronecker coefficients.
\newblock {\em Transactions of the American Mathematical Society},
  367(5):3647--3667, 2015.

\bibitem[BL18]{BL}
Jonah Blasiak and Ricky~Ini Liu.
\newblock {K}ronecker coefficients and noncommutative super {S}chur functions.
\newblock {\em Journal of Combinatorial Theory, Series A}, 158:315--361, 2018.

\bibitem[Bla17]{Bla}
Jonah Blasiak.
\newblock {K}ronecker coefficients for one hook shape.
\newblock {\em S{\'e}m. Lothar. Combin}, 77:2016--2017, 2017.

\bibitem[BMS15]{BMS}
Jonah Blasiak, Ketan Mulmuley, and Milind Sohoni.
\newblock {G}eometric {C}omplexity {T}heory {I}{V}: nonstandard quantum group
  for the {K}ronecker problem.
\newblock In {\em Memoirs of the American Mathematical Society}, volume 235.
  2015.

\bibitem[BO05]{BO05}
Cristina~M Ballantine and Rosa~C Orellana.
\newblock On the {K}ronecker product $s_{(n-p, p)}\ast s_{\lambda}$.
\newblock {\em The Electronic Journal of Combinatorics}, page R28, 2005.

\bibitem[BO06]{BO06}
Cristina~M Ballantine and Rosa~C Orellana.
\newblock A combinatorial interpretation for the coefficients in the
  {K}ronecker product $s_{(n-p, p)}\ast s_\lambda$.
\newblock {\em S{\'e}m. Lothar. Combin. A}, 54A:B54Af, 2006.

\bibitem[BOR09]{BOR09}
Emmanuel Briand, Rosa Orellana, and Mercedes Rosas.
\newblock Reduced {K}ronecker coefficients and counter--examples to
  {M}ulmuley's {S}trong {S}aturation {C}onjecture {S}{H}: With an appendix by
  {K}etan {M}ulmuley.
\newblock {\em Computational Complexity}, 18:577--600, 2009.

\bibitem[BOR11]{BOR2}
Emmanuel Briand, Rosa Orellana, and Mercedes Rosas.
\newblock The stability of the {K}ronecker product of {S}chur functions.
\newblock {\em Journal of Algebra}, 331(1):11--27, 2011.

\bibitem[Bri93]{Bri93}
Michel Brion.
\newblock Stable properties of plethysm: on two conjectures of {F}oulkes.
\newblock {\em Manuscripta mathematica}, 80:347--371, 1993.

\bibitem[Chr06]{Chr06}
Matthias Christandl.
\newblock {\em The Structure of Bipartite Quantum States - Insights from Group
  Theory and Cryptography}.
\newblock PhD thesis, 2006.
\newblock arXiv:guant-ph/0604183.

\bibitem[CR15]{CR}
Laura Colmenarejo and Mercedes Rosas.
\newblock Combinatorics on a family of reduced {K}ronecker coefficients.
\newblock {\em Comptes Rendus Math{\'e}matique}, 353:865--869, 2015.

\bibitem[Dvi93]{Dvi93}
Yoav Dvir.
\newblock On the {K}ronecker product of {$S_n$} characters.
\newblock {\em Journal of Algebra}, 154(1):125--140, 1993.

\bibitem[Ful97]{Ful97}
William Fulton.
\newblock {\em Young tableaux: with applications to representation theory and
  geometry}.
\newblock Number~35. Cambridge University Press, 1997.

\bibitem[GR85]{GR85}
Adriano~M Garsia and Jeffrey Remmel.
\newblock Shuffles of permutations and the {K}ronecker product.
\newblock {\em Graphs and Combinatorics}, 1:217--263, 1985.

\bibitem[Ike12]{Ike12}
Christian Ikenmeyer.
\newblock {\em Geometric complexity theory, tensor rank, and
  {L}ittlewood-{R}ichardson coefficients}.
\newblock PhD thesis, Paderborn, Universit{\"a}t Paderborn, 2012.

\bibitem[IMW17]{IMW17}
Christian Ikenmeyer, Ketan Mulmuley, and Michael Walter.
\newblock On vanishing of {K}ronecker coefficients.
\newblock {\em computational complexity}, 26:949--992, 2017.

\bibitem[IP17]{IP17}
Christian Ikenmeyer and Greta Panova.
\newblock Rectangular {K}ronecker coefficients and plethysms in geometric
  complexity theory.
\newblock {\em Advances in Mathematics}, 319:40--66, 2017.

\bibitem[IP22]{IP22}
Christian Ikenmeyer and Igor Pak.
\newblock What is in \#{P} and what is not?
\newblock In {\em 2022 IEEE 63rd Annual Symposium on Foundations of Computer
  Science (FOCS)}, pages 860--871. IEEE, 2022.

\bibitem[JK84]{JK}
Gordon James and Adalbert Kerber.
\newblock {\em The Representation Theory of the Symmetric Group}.
\newblock Encyclopedia of Mathematics and its Applications. Cambridge
  University Press, 1984.

\bibitem[Kir04]{Kir}
Anatol~N Kirillov.
\newblock An invitation to the generalized saturation conjecture.
\newblock {\em Publications of the Research Institute for Mathematical
  Sciences}, 40(4):1147--1239, 2004.

\bibitem[Kly04]{Kly}
Alexander Klyachko.
\newblock Quantum marginal problem and representations of the symmetric group.
\newblock arXiv:quant-ph/0409113, 2004.

\bibitem[KT99]{KT}
Allen Knutson and Terence Tao.
\newblock The honeycomb model of {$\GL_n(\IC)$} tensor products {I}: {P}roof of
  the saturation conjecture.
\newblock {\em Journal of the American Mathematical Society}, 12(4):1055--1090,
  1999.

\bibitem[Las79]{Las79}
Alain Lascoux.
\newblock Produit de {K}ronecker des repr{\'e}sentations du groupe
  sym{\'e}trique.
\newblock In {\em S{\'e}minaire d'Alg{\`e}bre Paul Dubreil et Marie-Paule
  Malliavin: Proceedings, Paris 1979 (32{\`e}me Ann{\'e}e)}, pages 319--329.
  Springer, 1979.

\bibitem[Liu17]{Liu17}
Ricky Liu.
\newblock A simplified {K}ronecker rule for one hook shape.
\newblock {\em Proceedings of the American Mathematical Society},
  145(9):3657--3664, 2017.

\bibitem[Mac98]{Mac}
Ian~Grant Macdonald.
\newblock {\em Symmetric functions and Hall polynomials}.
\newblock Oxford university press, 2. edition, 1998.

\bibitem[Man15]{Manivel}
Laurent Manivel.
\newblock On the asymptotics of {K}ronecker coefficients.
\newblock {\em Journal of Algebraic Combinatorics}, 42(4):999--1025, 2015.

\bibitem[Mur38]{Mur38}
Francis~D Murnaghan.
\newblock The analysis of the {K}ronecker product of irreducible
  representations of the symmetric group.
\newblock {\em American journal of mathematics}, 60(3):761--784, 1938.

\bibitem[Mur56]{Mur56}
Francis~D Murnaghan.
\newblock On the {K}ronecker product of irreducible representations of the
  symmetric group.
\newblock {\em Proceedings of the National Academy of Sciences}, 42(2):95--98,
  1956.

\bibitem[OZ20]{OZ20}
Rosa Orellana and Mike Zabrocki.
\newblock A combinatorial model for the decomposition of multivariate
  polynomial rings as {$S_n$}-modules.
\newblock {\em The Electronic Journal of Combinatorics}, 27:\#P3.24, 2020.

\bibitem[OZ21]{OZ21}
Rosa Orellana and Mike Zabrocki.
\newblock Symmetric group characters as symmetric functions.
\newblock {\em Advances in Mathematics}, 390, 2021.

\bibitem[Pak22]{Pak22}
Igor Pak.
\newblock What is a combinatorial interpretation?
\newblock preprint,
  \url{https://www.samuelfhopkins.com/OPAC/files/proceedings/pak.pdf}, to
  appear in Proc.\ OPAC, 2022.

\bibitem[Pan15]{Pporto}
Greta Panova.
\newblock {K}ronecker coefficients: combinatorics, complexity and beyond.
\newblock presentation slides,
  \url{https://drive.google.com/file/d/1T2bVbLa4Bozy2_VBzT_Z0aezw8Y7ysrX/view},
  2015.

\bibitem[Pan23]{Pan23}
Greta Panova.
\newblock Complexity and asymptotics of structure constants.
\newblock {\em Open Problems in Algebraic Combinatorics}, 2023.

\bibitem[PP14]{PP14}
Igor Pak and Greta Panova.
\newblock Unimodality via {K}ronecker products.
\newblock {\em Journal of Algebraic Combinatorics}, 40:1103--1120, 2014.

\bibitem[PP17a]{PP17}
I.~Pak and G.~Panova.
\newblock On the complexity of computing {K}ronecker coefficients.
\newblock {\em Comput.\ Complexity}, 26:1--36, 2017.

\bibitem[PP17b]{PPq}
Igor Pak and Greta Panova.
\newblock Bounds on certain classes of {K}ronecker and q-binomial coefficients.
\newblock {\em Journal of Combinatorial Theory, Series A}, 147:1--17, 2017.

\bibitem[PP20a]{PP20}
Igor Pak and Greta Panova.
\newblock Bounds on {K}ronecker coefficients via contingency tables.
\newblock {\em Linear Algebra and its Applications}, 602:157--178, 2020.

\bibitem[PP20b]{PPr}
Igor Pak and Greta Panova.
\newblock Breaking down the reduced {K}ronecker coefficients.
\newblock {\em Comptes Rendus Math{\'e}matique}, 358(4):463--468, 2020.

\bibitem[Rem89]{remm:89}
Jeffrey~B. Remmel.
\newblock A formula for the {K}ronecker products of {S}chur functions of hook
  shapes.
\newblock {\em J. Algebra}, 120(1):100--118, 1989.

\bibitem[RW94]{RW94}
Jeffrey~B. Remmel and Tamsen Whitehead.
\newblock On the {K}ronecker product of {S}chur functions of two row shapes.
\newblock {\em Bull. Belg. Math. Soc. Simon Stevin}, 1(5):649--683, 1994.

\bibitem[Sag13]{Sag}
Bruce~E Sagan.
\newblock {\em The symmetric group: representations, combinatorial algorithms,
  and symmetric functions}, volume 203.
\newblock Springer Science \& Business Media, 2013.

\bibitem[SS16]{SS}
Steven Sam and Andrew Snowden.
\newblock Proof of {S}tembridge’s conjecture on stability of {K}ronecker
  coefficients.
\newblock {\em Journal of Algebraic Combinatorics}, 43(1):1--10, 2016.

\bibitem[Sta99]{EC2}
Richard Stanley.
\newblock {\em Enumerative Combinatorics}, volume~2.
\newblock Cambridge University Press, 1 edition, 1999.

\bibitem[Sta00]{Sta00}
Richard~P Stanley.
\newblock Positivity {P}roblems and {C}onjectures in {A}lgebraic
  {C}ombinatorics.
\newblock {\em Mathematics: Frontiers and Perspectives}, pages 295--319, 2000.

\bibitem[Ste]{Ste}
John~R. Stembridge.
\newblock Generalized stability of {K}ronecker coefficients.
\newblock \url{http://www.math.lsa.umich.edu/~jrs}.

\bibitem[Tew15]{T15}
Vasu~V Tewari.
\newblock {K}ronecker coefficients for some near-rectangular partitions.
\newblock {\em Journal of Algebra}, 429:287--317, 2015.

\bibitem[Val99]{Val99}
Ernesto Vallejo.
\newblock Stability of {K}ronecker products of irreducible characters of the
  symmetric group.
\newblock {\em The Electronic Journal of Combinatorics}, 6:R39, 1999.

\bibitem[Val09]{Val09}
Ernesto Vallejo.
\newblock A stability property for coefficients in {K}ronecker products of
  complex {$S_n$} characters.
\newblock {\em The Electronic Journal of Combinatorics}, 16(22):1, 2009.

\bibitem[Val20]{Val20}
Ernesto Vallejo.
\newblock Stability of {K}ronecker coefficients via discrete tomography.
\newblock {\em Discrete Mathematics}, 343(5):111817, 2020.

\end{thebibliography}
}
\end{document}